\documentclass[12pt,fleqn]{article}

\usepackage{amsmath,amssymb,amsthm}
\usepackage{geometry}
\usepackage[pdfpagemode=UseNone,pdfstartview=FitH]{hyperref}

\newcommand{\abs}[1]{\left|#1\right|}
\newcommand{\bdry}[1]{\partial #1}
\newcommand{\closure}[1]{\overline{#1}}
\newcommand{\dint}{\ds{\int}}
\newcommand{\ds}[1]{\displaystyle #1}
\newcommand{\eps}{\varepsilon}
\newcommand{\half}{\frac{1}{2}}
\newcommand{\norm}[2][]{\left\|#2\right\|_{#1}}
\renewcommand{\O}{\text{O}}
\renewcommand{\o}{\text{o}}
\newcommand{\PS}[1]{$(\text{PS})_{#1}$}
\newcommand{\R}{\mathbb R}
\newcommand{\seq}[1]{\left(#1\right)}
\newcommand{\set}[1]{\left\{#1\right\}}
\newcommand{\vol}[1]{\left|#1\right|}

\newenvironment{enumroman}{\begin{enumerate}

}{\end{enumerate}}

\newtheorem{lemma}{Lemma}[section]
\newtheorem{theorem}[lemma]{Theorem}

\numberwithin{equation}{section}

\title{\bf On a class of semipositone problems with singular Trudinger-Moser nonlinearities}
\author{\bf Shiqiu Fu and Kanishka Perera\\
Department of Mathematical Sciences\\
Florida Institute of Technology\\
Melbourne, FL 32901, USA\\
\em sfu2013@my.fit.edu \& kperera@fit.edu}
\date{}

\begin{document}

\maketitle

\begin{abstract}
We prove the existence of positive solutions for a class of semipositone problem with singular Trudinger-Moser nonlinearities. The proof is based on compactness and regularity arguments.
\end{abstract}

\section{Introduction}

Let $\Omega$ be a bounded domain in $\R^N,\, N \ge 2$ and let $f$ be a Carath\'{e}odory function on $\Omega \times [0,\infty)$. The semilinear elliptic boundary value problem
\[
\left\{\begin{aligned}
- \Delta u & = f(x,u) && \text{in } \Omega\\[5pt]
u & > 0 && \text{in } \Omega\\[5pt]
u & = 0 && \text{on } \bdry{\Omega}
\end{aligned}\right.
\]
is said to be of semipositone type if $f(\cdot,0) < 0$ on a set of positive measure. It is notoriously difficult to find positive solutions of this class of problems due to the fact that $u = 0$ is not a subsolution (see, e.g., Castro and Shivaji \cite{MR943804}, Ali et al.\! \cite{MR1116249}, Ambrosetti et al.\! \cite{MR1270096}, Chhetri et al.\! \cite{MR3406455}, Castro et al.\! \cite{MR3507282}, Costa et al.\! \cite{MR3626519}, and their references).

The purpose of the present paper is to study a class of semipositone problems with singular exponential nonlinearities in dimension $N = 2$. We consider the problem
\begin{equation} \label{1}
\left\{\begin{aligned}
- \Delta u & = \lambda u\, \frac{e^{\alpha u^2}}{|x|^\gamma} + \mu\, g(u) && \text{in } \Omega\\
u & > 0 && \text{in } \Omega\\[5pt]
u & = 0 && \text{on } \bdry{\Omega},
\end{aligned}\right.
\end{equation}
where $\Omega$ is a smooth bounded domain in $\R^2$ containing the origin, $\alpha > 0$, $0 \le \gamma < 2$, $\lambda, \mu > 0$ are parameters, and $g$ is a continuous function on $[0,\infty)$ satisfying
\begin{equation} \label{2}
\lim_{t \to \infty}\, \frac{g(t)}{e^{\beta t^2}} = 0 \quad \forall \beta > 0
\end{equation}
and
\begin{equation} \label{31}
\sup_{t \in [0,\infty)}\, \big(2G(t) - tg(t)\big) < \infty,
\end{equation}
where $G(t) = \int_0^t g(s)\, ds$. We make no assumptions about the sign of $g(0)$ and hence allow the semipositone case $g(0) < 0$. For example, the functions $g(t) = -1$, $g(t) = t^p - 1$, where $p \ge 1$, and $g(t) = e^t - 2$ all satisfy \eqref{2}, \eqref{31}, and $g(0) < 0$.

The motivation for problem \eqref{1} comes from the following singular Trudinger-Moser embedding of Adimurthi and Sandeep \cite{MR2329019}:
\[
\int_\Omega \frac{e^{\alpha u^2}}{|x|^\gamma}\, dx < \infty \quad \forall u \in H^1_0(\Omega)
\]
for all $\alpha > 0$ and $0 \le \gamma < 2$, and
\begin{equation} \label{32}
\sup_{\norm[H^1_0(\Omega)]{u} \le 1}\, \int_\Omega \frac{e^{\alpha u^2}}{|x|^\gamma}\, dx < \infty
\end{equation}
if and only if $\alpha/4 \pi + \gamma/2 \le 1$. Our problem is critical with respect to this embedding and hence the variational functional associated with this problem lacks compactness, which is an additional difficulty in finding solutions.

Let $\lambda_1(\gamma) > 0$ be the first eigenvalue of the singular eigenvalue problem
\[
\left\{\begin{aligned}
- \Delta u & = \lambda\, \frac{u}{|x|^\gamma} && \text{in } \Omega\\
u & = 0 && \text{on } \bdry{\Omega}
\end{aligned}\right.
\]
given by
\begin{equation} \label{3}
\lambda_1(\gamma) = \inf_{u \in H^1_0(\Omega) \setminus \set{0}}\, \frac{\dint_\Omega |\nabla u|^2\, dx}{\dint_\Omega \frac{u^2}{|x|^\gamma}\, dx}.
\end{equation}
We will show that problem \eqref{1} has a positive solution for all $0 < \lambda < \lambda_1(\gamma)$ and $\mu > 0$ sufficiently small. We have the following theorem.

\begin{theorem} \label{Theorem 1}
Assume that $\alpha > 0$ and $0 \le \gamma < 1$ satisfy
\[
\frac{\alpha}{4 \pi} + \frac{\gamma}{2} \le 1,
\]
$0 < \lambda < \lambda_1(\gamma)$, and $g$ satisfies \eqref{2} and \eqref{31}. Then there exists a $\mu^\ast > 0$ such that for all $0 < \mu < \mu^\ast$, problem \eqref{1} has a solution $u_\mu$.
\end{theorem}

We note that this result does not follow from standard arguments based on the maximum principle since $g(0)$ is not assumed to be nonnegative. Our proof is based on regularity arguments and will be given in Section \ref{Proof}, after establishing a suitable compactness property of an associated variational functional in the next section.

\section{A compactness result}

In this section we consider the modified problem
\begin{equation} \label{30}
\left\{\begin{aligned}
- \Delta u & = \lambda u^+\, \frac{e^{\alpha\, (u^+)^2}}{|x|^\gamma} + \mu\, \widetilde{g}(u) && \text{in } \Omega\\
u & = 0 && \text{on } \bdry{\Omega},
\end{aligned}\right.
\end{equation}
where $u^+(x) = \max \set{u(x),0}$ and
\[
\widetilde{g}(t) = \begin{cases}
0, & t \le -1\\[5pt]
(1 + t)\, g(0), & -1 < t < 0\\[5pt]
g(t), & t \ge 0.
\end{cases}
\]
Weak solutions of this problem coincide with critical points of the $C^1$-functional
\[
E_\mu(u) = \int_\Omega \left[\half\, |\nabla u|^2 - \frac{\lambda}{2 \alpha}\, \frac{e^{\alpha\, (u^+)^2} - 1}{|x|^\gamma} - \mu\, \widetilde{G}(u)\right] dx, \quad u \in H^1_0(\Omega),
\]
where $\widetilde{G}(t) = \int_0^t \widetilde{g}(s)\, ds$. The main result of this section is the following compactness result.

\begin{theorem} \label{Theorem 2}
Assume that $\alpha > 0$ and $0 \le \gamma < 2$ satisfy $\alpha/4 \pi + \gamma/2 \le 1$ and $g$ satisfies \eqref{2} and \eqref{31}. If $\mu_j > 0,\, \mu_j \to \mu \ge 0$, $\seq{u_j} \subset H^1_0(\Omega)$, and
\[
E_{\mu_j}(u_j) \to c, \qquad E_{\mu_j}'(u_j) \to 0
\]
for some $c \ne 0$ satisfying
\begin{equation} \label{4}
c < \frac{2 \pi}{\alpha} \left(1 - \frac{\gamma}{2}\right) - \frac{\mu \theta}{2} \vol{\Omega},
\end{equation}
where
\[
\theta = \sup_{t \in \R}\, \big(2 \widetilde{G}(t) - t \widetilde{g}(t)\big)
\]
and $\vol{\cdot}$ denotes the Lebesgue measure in $\R^2$, then a subsequence of $\seq{u_j}$ converges to a critical point of $E_\mu$ at the level $c$. In particular, $E_\mu$ satisfies the {\em \PS{c}} condition for all $c \ne 0$ satisfying \eqref{4}.
\end{theorem}

First we prove the following lemma.

\begin{lemma} \label{Lemma 4}
If $\seq{u_j}$ is a sequence in $H^1_0(\Omega)$ converging a.e.\! to $u \in H^1_0(\Omega)$ and
\begin{equation} \label{5}
\sup_j \int_\Omega (u_j^+)^2\, \frac{e^{\alpha\, (u_j^+)^2}}{|x|^\gamma}\, dx < \infty,
\end{equation}
then
\[
\int_\Omega \frac{e^{\alpha\, (u_j^+)^2}}{|x|^\gamma}\, dx \to \int_\Omega \frac{e^{\alpha\, (u^+)^2}}{|x|^\gamma}\, dx.
\]
\end{lemma}

\begin{proof}
For $M > 0$, write
\[
\int_\Omega \frac{e^{\alpha\, (u_j^+)^2}}{|x|^\gamma}\, dx = \int_{\{u_j^+ < M\}} \frac{e^{\alpha\, (u_j^+)^2}}{|x|^\gamma}\, dx + \int_{\{u_j^+ \ge M\}} \frac{e^{\alpha\, (u_j^+)^2}}{|x|^\gamma}\, dx.
\]
By \eqref{5},
\[
\int_{\{u_j^+ \ge M\}} \frac{e^{\alpha\, (u_j^+)^2}}{|x|^\gamma}\, dx \le \frac{1}{M^2} \int_\Omega (u_j^+)^2\, \frac{e^{\alpha\, (u_j^+)^2}}{|x|^\gamma}\, dx = \O\! \left(\frac{1}{M^2}\right) \text{ as } M \to \infty.
\]
Hence
\[
\int_\Omega \frac{e^{\alpha\, (u_j^+)^2}}{|x|^\gamma}\, dx = \int_{\{u_j^+ < M\}} \frac{e^{\alpha\, (u_j^+)^2}}{|x|^\gamma}\, dx + \O\! \left(\frac{1}{M^2}\right),
\]
and the conclusion follows by first letting $j \to \infty$ and then letting $M \to \infty$.
\end{proof}

We will also need the following result from Adimurthi and Sandeep \cite[Theorem 2.3]{MR2329019}.

\begin{lemma} \label{Lemma 5}
Let $0 \le \gamma < 2$. If $\seq{u_j}$ is a sequence in $H^1_0(\Omega)$ with $\norm{u_j} = 1$ for all $j$ and converging weakly to a nonzero function $u$, then
\[
\sup_j \int_\Omega \frac{e^{\beta u_j^2}}{|x|^\gamma}\, dx < \infty
\]
for all $\beta < 4 \pi (1 - \gamma/2)/(1 - \norm{u}^2)$.
\end{lemma}

We are now ready to prove Theorem \ref{Theorem 2}.

\begin{proof}[Proof of Theorem \ref{Theorem 2}]
We have
\begin{equation} \label{6}
E_{\mu_j}(u_j) = \half \norm{u_j}^2 - \frac{\lambda}{2 \alpha} \int_\Omega \frac{e^{\alpha\, (u_j^+)^2} - 1}{|x|^\gamma}\, dx - \mu_j \int_\Omega \widetilde{G}(u_j)\, dx = c + \o(1)
\end{equation}
and
\begin{equation} \label{7}
E_{\mu_j}'(u_j)\, u_j = \norm{u_j}^2 - \lambda \int_\Omega (u_j^+)^2\, \frac{e^{\alpha\, (u_j^+)^2}}{|x|^\gamma}\, dx - \mu_j \int_\Omega u_j\, \widetilde{g}(u_j)\, dx = \o(\norm{u_j}).
\end{equation}
Multiplying \eqref{6} by $4$ and subtracting \eqref{7} gives
\begin{multline*}
\norm{u_j}^2 + \lambda \int_\Omega \left(\left[(u_j^+)^2 - \frac{2}{\alpha}\right] e^{\alpha\, (u_j^+)^2} + \frac{2}{\alpha}\right) \frac{dx}{|x|^\gamma} + \mu_j \int_\Omega \big(u_j\, \widetilde{g}(u_j) - 4 \widetilde{G}(u_j)\big)\, dx\\[5pt]
= 4c + \o(\norm{u_j} + 1),
\end{multline*}
and this together with \eqref{2} implies that $\seq{u_j}$ is bounded in $H^1_0(\Omega)$. Hence a renamed subsequence converges to some $u$ weakly in $H^1_0(\Omega)$, strongly in $L^p(\Omega)$ for all $p \in [1,\infty)$, and a.e.\! in $\Omega$. Moreover,
\[
\sup_j \int_\Omega e^{\beta u_j^2}\, dx < \infty
\]
for all $\beta \le 4 \pi/(\sup_j \norm{u_j})$ by \eqref{32}, and hence $\int_\Omega u_j\, \widetilde{g}(u_j)\, dx$ is bounded by \eqref{2}. Then
\begin{equation} \label{9}
\sup_j \int_\Omega (u_j^+)^2\, \frac{e^{\alpha\, (u_j^+)^2}}{|x|^\gamma}\, dx < \infty
\end{equation}
by \eqref{7}, and hence
\begin{equation} \label{10}
\int_\Omega \frac{e^{\alpha\, (u_j^+)^2}}{|x|^\gamma}\, dx \to \int_\Omega \frac{e^{\alpha\, (u^+)^2}}{|x|^\gamma}\, dx
\end{equation}
by Lemma \ref{Lemma 4}. Denoting by $C$ a generic positive constant,
\[
|u_j\, \widetilde{g}(u_j)| \le |u_j|\, \big(e^{\alpha\, (u_j^+)^2/2} + C\big) \le \frac{e^{\alpha\, (u_j^+)^2}}{|x|^\gamma} + C \left(u_j^2 + 1\right)
\]
by \eqref{2}, so it follows from \eqref{10} and the dominated convergence theorem that
\begin{equation} \label{11}
\int_\Omega u_j\, \widetilde{g}(u_j)\, dx \to \int_\Omega u\, \widetilde{g}(u)\, dx.
\end{equation}
Similarly,
\begin{equation} \label{33}
\int_\Omega \widetilde{G}(u_j)\, dx \to \int_\Omega \widetilde{G}(u)\, dx.
\end{equation}

We claim that the weak limit $u$ is nonzero. Suppose $u = 0$. Then
\begin{equation} \label{12}
\int_\Omega \frac{e^{\alpha\, (u_j^+)^2}}{|x|^\gamma}\, dx \to \int_\Omega \frac{dx}{|x|^\gamma}, \qquad \int_\Omega u_j\, \widetilde{g}(u_j)\, dx \to 0, \qquad \int_\Omega \widetilde{G}(u_j)\, dx \to 0
\end{equation}
by \eqref{10}--\eqref{33}. So \eqref{6} implies that $c \ge 0$ and
\begin{equation} \label{34}
\norm{u_j} \to (2c)^{1/2}.
\end{equation}
Noting that $c < 2 \pi\, (1 - \gamma/2)/\alpha$ by \eqref{4}, let $2c < \nu < 4 \pi\, (1 - \gamma/2)/\alpha$. Then \eqref{34} implies that $\norm{u_j} \le \nu^{1/2}$ for all $j \ge j_0$ for some $j_0$. Let $q = 4 \pi\, (1 - \gamma/2)/\alpha \nu > 1$ and let $1/(1 - 1/q) < r < 2/\gamma\, (1 - 1/q)$. By the H\"{o}lder inequality,
\[
\int_\Omega (u_j^+)^2\, \frac{e^{\alpha\, (u_j^+)^2}}{|x|^\gamma}\, dx \le \left(\int_\Omega |u_j|^{2p}\, dx\right)^{1/p} \left(\int_\Omega \frac{e^{q \alpha u_j^2}}{|x|^\gamma}\, dx\right)^{1/q} \left(\int_\Omega \frac{dx}{|x|^{\gamma r\, (1-1/q)}}\right)^{1/r},
\]
where $1/p + 1/q + 1/r = 1$. The first integral on the right-hand side converges to zero since $u = 0$, the second integral is bounded for $j \ge j_0$ by \eqref{32} since $q \alpha u_j^2 = 4 \pi\, (1 - \gamma/2)\, \widetilde{u}_j^2$, where $\widetilde{u}_j = u_j/\nu^{1/2}$ satisfies $\norm{\widetilde{u}_j} \le 1$, and the last integral is finite since $\gamma r\, (1 - 1/q) < 2$, so
\[
\int_\Omega (u_j^+)^2\, \frac{e^{\alpha\, (u_j^+)^2}}{|x|^\gamma}\, dx \to 0.
\]
Then $u_j \to 0$ by \eqref{7} and \eqref{12}, and hence $c = 0$ by \eqref{34}, a contradiction. So $u$ is nonzero.

Since $E_{\mu_j}'(u_j) \to 0$,
\begin{equation} \label{35}
\int_\Omega \nabla u_j \cdot \nabla v\, dx - \lambda \int_\Omega u_j^+\, \frac{e^{\alpha\, (u_j^+)^2}}{|x|^\gamma}\, v\, dx - \mu_j \int_\Omega \widetilde{g}(u_j)\, v\, dx \to 0
\end{equation}
for all $v \in H^1_0(\Omega)$. For $v \in C^\infty_0(\Omega)$, an argument similar to that in the proof of Lemma \ref{Lemma 4} using the estimate
\[
\abs{\int_{\{u_j^+ \ge M\}} u_j^+\, \frac{e^{\alpha\, (u_j^+)^2}}{|x|^\gamma}\, v\, dx} \le \frac{\sup |v|}{M} \int_\Omega (u_j^+)^2\, \frac{e^{\alpha\, (u_j^+)^2}}{|x|^\gamma}\, dx
\]
and \eqref{9} shows that $\dint_\Omega u_j^+\, \frac{e^{\alpha\, (u_j^+)^2}}{|x|^\gamma}\, v\, dx \to \dint_\Omega u^+\, \frac{e^{\alpha\, (u^+)^2}}{|x|^\gamma}\, v\, dx$. Moreover, denoting by $C$ a generic positive constant,
\[
|\widetilde{g}(u_j)\, v| \le \sup |v|\, \big(e^{\alpha\, (u_j^+)^2} + C\big) \le C\, \sup |v| \left(\frac{e^{\alpha\, (u_j^+)^2}}{|x|^\gamma} + 1\right)
\]
by \eqref{2}, so it follows from \eqref{10} and the dominated convergence theorem that
\[
\int_\Omega \widetilde{g}(u_j)\, v\, dx \to \int_\Omega \widetilde{g}(u)\, v\, dx.
\]
So it follows from \eqref{35} that
\[
\int_\Omega \nabla u \cdot \nabla v\, dx = \lambda \int_\Omega u^+\, \frac{e^{\alpha\, (u^+)^2}}{|x|^\gamma}\, v\, dx + \mu \int_\Omega \widetilde{g}(u)\, v\, dx.
\]
Then this holds for all $v \in H^1_0(\Omega)$ by density, and taking $v = u$ gives
\begin{equation} \label{13}
\norm{u}^2 = \lambda \int_\Omega (u^+)^2\, \frac{e^{\alpha\, (u^+)^2}}{|x|^\gamma}\, dx + \mu \int_\Omega u\, \widetilde{g}(u)\, dx.
\end{equation}

Next we claim that
\begin{equation} \label{14}
\int_\Omega (u_j^+)^2\, \frac{e^{\alpha\, (u_j^+)^2}}{|x|^\gamma}\, dx \to \int_\Omega (u^+)^2\, \frac{e^{\alpha\, (u^+)^2}}{|x|^\gamma}\, dx.
\end{equation}
We have
\begin{equation} \label{15}
(u_j^+)^2\, \frac{e^{\alpha\, (u_j^+)^2}}{|x|^\gamma} \le u_j^2\, \frac{e^{\alpha u_j^2}}{|x|^\gamma} = u_j^2\, \frac{ e^{\alpha\, \norm{u_j}^2\, \widetilde{u}_j^2}}{|x|^\gamma},
\end{equation}
where $\widetilde{u}_j = u_j/\norm{u_j}$. Setting
\[
\kappa = \frac{\lambda}{2 \alpha} \int_\Omega \frac{e^{\alpha\, (u^+)^2} - 1}{|x|^\gamma}\, dx + \mu \int_\Omega \widetilde{G}(u)\, dx,
\]
we have
\[
\norm{u_j}^2 \to 2\, (c + \kappa)
\]
by \eqref{6}, \eqref{10}, and \eqref{33}, so $\widetilde{u}_j$ converges weakly and a.e.\! to $\widetilde{u} = u/[2\, (c + \kappa)]^{1/2}$. Then
\begin{equation} \label{16}
\norm{u_j}^2 \left(1 - \norm{\widetilde{u}}^2\right) \to 2\, (c + \kappa) - \norm{u}^2.
\end{equation}
Since $te^t \ge e^t - 1$ for all $t \ge 0$,
\[
\int_\Omega (u^+)^2\, \frac{e^{\alpha\, (u^+)^2}}{|x|^\gamma}\, dx \ge \frac{1}{\alpha} \int_\Omega \frac{e^{\alpha\, (u^+)^2} - 1}{|x|^\gamma}\, dx,
\]
and
\[
\int_\Omega u\, \widetilde{g}(u)\, dx \ge 2 \int_\Omega \widetilde{G}(u)\, dx - \theta \vol{\Omega}
\]
since $\theta \ge 2 \widetilde{G}(t) - t \widetilde{g}(t)$ for all $t \in \R$, so it follows from \eqref{13} that $\norm{u}^2 \ge 2 \kappa - \mu \theta \vol{\Omega}$. Hence
\begin{equation} \label{17}
2\, (c + \kappa) - \norm{u}^2 \le 2c + \mu \theta \vol{\Omega} < \frac{4 \pi}{\alpha} \left(1 - \frac{\gamma}{2}\right)
\end{equation}
by \eqref{4}. We are done if $\norm{\widetilde{u}} = 1$, so suppose $\norm{\widetilde{u}} < 1$ and let
\[
\frac{2c + \mu \theta \vol{\Omega}}{1 - \norm{\widetilde{u}}^2} < \widetilde{\nu} - 2 \eps < \widetilde{\nu} < \frac{4 \pi\, (1 - \gamma/2)/\alpha}{1 - \norm{\widetilde{u}}^2}.
\]
Then $\norm{u_j}^2 \le \widetilde{\nu} - 2 \eps$ for all $j \ge j_0$ for some $j_0$ by \eqref{16} and \eqref{17}, and
\begin{equation} \label{18}
\sup_j \int_\Omega \frac{e^{\alpha \widetilde{\nu}\, \widetilde{u}_j^2}}{|x|^\gamma}\, dx < \infty
\end{equation}
by Lemma \ref{Lemma 5}. For $M > 0$ and $j \ge j_0$, \eqref{15} then gives
\begin{align*}
& \phantom{\le \text{ }} \int_{\{u_j^+ \ge M\}} (u_j^+)^2\, \frac{e^{\alpha\, (u_j^+)^2}}{|x|^\gamma}\, dx\\[2pt]
& \le \int_{\{u_j^+ \ge M\}} u_j^2\, \frac{e^{\alpha\, (\widetilde{\nu} - 2 \eps)\, \widetilde{u}_j^2}}{|x|^\gamma}\, dx\\[2pt]
& = \norm{u_j}^2 \int_{\{u_j^+ \ge M\}} \widetilde{u}_j^2\, e^{- \eps \alpha\, \widetilde{u}_j^2}\, e^{- \eps \alpha\, (u_j/\norm{u_j})^2}\, \frac{e^{\alpha \widetilde{\nu}\, \widetilde{u}_j^2}}{|x|^\gamma}\, dx\\[2pt]
& \le \left(\max_{t \ge 0}\, te^{- \eps \alpha\, t}\right) \norm{u_j}^2 e^{- \eps \alpha\, (M/\norm{u_j})^2} \int_\Omega \frac{e^{\alpha \widetilde{\nu}\, \widetilde{u}_j^2}}{|x|^\gamma}\, dx.
\end{align*}
The last expression goes to zero as $M \to \infty$ uniformly in $j$ since $\norm{u_j}$ is bounded and \eqref{18} holds, so \eqref{14} now follows as in the proof of Lemma \ref{Lemma 4}.

Now it follows from \eqref{7}, \eqref{14}, \eqref{11}, and \eqref{13} that
\[
\norm{u_j}^2 \to \lambda \int_\Omega (u^+)^2\, \frac{e^{\alpha\, (u^+)^2}}{|x|^\gamma}\, dx + \mu \int_\Omega u\, \widetilde{g}(u)\, dx = \norm{u}^2
\]
and hence $\norm{u_j} \to \norm{u}$, so $u_j \to u$. Clearly, $E_\mu(u) = c$ and $E_\mu'(u) = 0$.
\end{proof}

\section{Proof of Theorem \ref{Theorem 1}} \label{Proof}

In this section we prove our main result. By Theorem \ref{Theorem 2}, $E_\mu$ satisfies the \PS{c} condition for all $c \ne 0$ satisfying
\[
c < \frac{2 \pi}{\alpha} \left(1 - \frac{\gamma}{2}\right) - \frac{\mu \theta}{2} \vol{\Omega}.
\]
First we show that $E_\mu$ has a uniformly positive mountain pass level below this threshold for compactness for all sufficiently small $\mu > 0$. Take $r > 0$ so small that $\closure{B_r(0)} \subset \Omega$ and let
\[
v_j(x) = \frac{1}{\sqrt{2 \pi}}\, \begin{cases}
\sqrt{\log j}, & |x| \le r/j\\[10pt]
\dfrac{\log (r/|x|)}{\sqrt{\log j}}, & r/j < |x| < r\\[10pt]
0, & |x| \ge r.
\end{cases}
\]
It is easily seen that $v_j \in H^1_0(\Omega)$ with $\norm{v_j} = 1$ and
\begin{equation} \label{20}
\int_\Omega v_j^2\, dx = \O(1/\log j) \quad \text{as } j \to \infty.
\end{equation}

\begin{lemma} \label{Lemma 6}
There exist $\mu_0, \rho, c_0 > 0$, $j_0 \ge 2$, $R > \rho$, and $\vartheta < \dfrac{2 \pi}{\alpha} \left(1 - \dfrac{\gamma}{2}\right)$ such that the following hold for all $\mu \in (0,\mu_0)$:
\begin{enumroman}
\item \label{Lemma 6 (i)} $\norm{u} = \rho \implies E_\mu(u) \ge c_0$,
\item \label{Lemma 6 (ii)} $E_\mu(Rv_{j_0}) \le 0$,
\item \label{Lemma 6 (iii)} denoting by $\Gamma = \set{\gamma \in C([0,1],H^1_0(\Omega)) : \gamma(0) = 0,\, \gamma(1) = Rv_{j_0}}$ the class of paths joining the origin to $Rv_{j_0}$,
    \begin{equation} \label{21}
    c_0 \le c_\mu := \inf_{\gamma \in \Gamma}\, \max_{u \in \gamma([0,1])}\, E_\mu(u) \le \vartheta + C \mu^2
    \end{equation}
    for some constant $C > 0$,
\item \label{Lemma 6 (iv)} $E_\mu$ has a critical point $u_\mu$ at the level $c_\mu$.
\end{enumroman}
\end{lemma}

\begin{proof}
Set $\rho = \norm{u}$ and $\widetilde{u} = u/\rho$. Since $e^t - 1 \le t + t^2 e^t$ for all $t \ge 0$,
\begin{equation} \label{36}
\frac{1}{\alpha} \int_\Omega \frac{e^{\alpha\, (u^+)^2} - 1}{|x|^\gamma}\, dx \le \int_\Omega \frac{u^2}{|x|^\gamma}\, dx + \alpha \int_\Omega u^4\, \frac{e^{\alpha u^2}}{|x|^\gamma}\, dx.
\end{equation}
By \eqref{3},
\begin{equation}
\int_\Omega \frac{u^2}{|x|^\gamma}\, dx \le \frac{\rho^2}{\lambda_1(\gamma)}.
\end{equation}
Let $2 < r < 4/\gamma$. By the H\"{o}lder inequality,
\begin{equation} \label{37}
\int_\Omega u^4\, \frac{e^{\alpha u^2}}{|x|^\gamma}\, dx \le \left(\int_\Omega u^{4p}\, dx\right)^{1/p} \left(\int_\Omega \frac{e^{2 \alpha u^2}}{|x|^\gamma}\, dx\right)^{1/2} \left(\int_\Omega \frac{dx}{|x|^{\gamma r/2}}\right)^{1/r},
\end{equation}
where $1/p + 1/r = 1/2$. The first integral on the right-hand side is bounded by $C \rho^4$ for some constant $C > 0$ by the Sobolev embedding. Since $2 \alpha u^2 = 2 \alpha \rho^2\, \widetilde{u}^2$ and $\norm{\widetilde{u}} = 1$, the second integral is bounded when $\rho^2 \le 2 \pi\, (1 - \gamma/2)/\alpha$ by \eqref{32}. The last integral is finite since $\gamma r < 4$. So combining \eqref{36}--\eqref{37} gives
\[
\frac{1}{\alpha} \int_\Omega \frac{e^{\alpha\, (u^+)^2} - 1}{|x|^\gamma}\, dx \le \frac{\rho^2}{\lambda_1(\gamma)} + \O(\rho^4) \quad \text{as } \rho \to 0.
\]
On the other hand, it follows from \eqref{2} that $\dint_\Omega \widetilde{G}(u)\, dx$ is bounded on bounded subsets of $H^1_0(\Omega)$. So
\[
E_\mu(u) \ge \half \left(1 - \frac{\lambda}{\lambda_1(\gamma)}\right) \rho^2 + \O(\rho^4) - C \mu \quad \text{as } \rho \to 0
\]
for some constant $C > 0$. Since $\lambda(\gamma) < \lambda_1$, \ref{Lemma 6 (i)} follows from this for sufficiently small $\rho, \mu, c_0 > 0$.

Since $\norm{v_j} = 1$ and $v_j \ge 0$,
\[
E_\mu(tv_j) = \frac{t^2}{2} - \int_\Omega \left[\frac{\lambda}{2 \alpha}\, \frac{e^{\alpha t^2 v_j^2} - 1}{|x|^\gamma} + \mu\, G(tv_j)\right] dx
\]
for $t \ge 0$. For $\mu \le \lambda/2$, this gives
\[
E_\mu(tv_j) \le \frac{t^2}{2} - \int_\Omega \left[\frac{\lambda}{4 \alpha}\, \frac{e^{\alpha t^2 v_j^2} - 1}{|x|^\gamma} + \mu\, F(x,tv_j)\right] dx,
\]
where
\[
F(x,t) = \frac{1}{2 \alpha}\, \frac{e^{\alpha t^2} - 1}{|x|^\gamma} + G(t) = \int_0^t \left(s\, \frac{e^{\alpha s^2}}{|x|^\gamma} + g(s)\right) ds \ge - Ct
\]
for some generic positive constant $C$ by \eqref{2}, so
\[
E_\mu(tv_j) \le \frac{t^2}{2} - \frac{\lambda}{4 \alpha} \int_\Omega \frac{e^{\alpha t^2 v_j^2} - 1}{|x|^\gamma}\, dx + C \mu t \int_\Omega v_j\, dx.
\]
Since
\[
C \mu t \int_\Omega v_j\, dx \le C \mu t \left(\int_\Omega v_j^2\, dx\right)^{1/2} \le C \mu^2 + \frac{t^2}{2} \int_\Omega v_j^2\, dx,
\]
then
\[
E_\mu(tv_j) \le H_j(t) + C \mu^2,
\]
where
\[
H_j(t) = \frac{t^2}{2} \left(1 + \int_\Omega v_j^2\, dx\right) - \frac{\lambda}{4 \alpha} \int_\Omega \frac{e^{\alpha t^2 v_j^2} - 1}{|x|^\gamma}\, dx \to - \infty \quad \text{as } t \to \infty.
\]
So to prove \ref{Lemma 6 (ii)} and \ref{Lemma 6 (iii)}, it suffices to show that $\exists j_0 \ge 2$ such that
\[
\vartheta := \sup_{t \ge 0}\, H_{j_0}(t) < \frac{2 \pi}{\alpha} \left(1 - \frac{\gamma}{2}\right).
\]

Suppose $\sup_{t \ge 0} H_j(t) \ge 2 \pi\, (1 - \gamma/2)/\alpha$ for all $j$. Since $H_j(t) \to - \infty$ as $t \to \infty$, there exists $t_j > 0$ such that
\begin{equation} \label{23}
H_j(t_j) = \frac{t_j^2}{2}\, (1 + \eps_j) - \frac{\lambda}{4 \alpha} \int_\Omega \frac{e^{\alpha t_j^2 v_j^2} - 1}{|x|^\gamma}\, dx = \sup_{t \ge 0}\, H_j(t) \ge \frac{2 \pi}{\alpha} \left(1 - \frac{\gamma}{2}\right)
\end{equation}
and
\begin{equation} \label{24}
H_j'(t_j) = t_j \left(1 + \eps_j - \frac{\lambda}{2} \int_\Omega v_j^2\, \frac{e^{\alpha t_j^2 v_j^2}}{|x|^\gamma}\, dx\right) = 0,
\end{equation}
where $\eps_j = \dint_\Omega v_j^2\, dx \to 0$ by \eqref{20}. The inequality in \eqref{23} gives
\[
\alpha t_j^2 \ge \frac{4 \pi}{1 + \eps_j} \left(1 - \frac{\gamma}{2}\right),
\]
and then \eqref{24} gives
\begin{multline*}
\frac{2}{\lambda}\, (1 + \eps_j) = \int_\Omega v_j^2\, \frac{e^{\alpha t_j^2 v_j^2}}{|x|^\gamma}\, dx \ge \int_{B_{r/j}(0)} v_j^2\, \frac{e^{4 \pi\, (1 - \gamma/2)\, v_j^2/(1 + \eps_j)}}{|x|^\gamma}\, dx\\[5pt]
= \frac{r^{2\, (1 - \gamma/2)}}{2\, (1 - \gamma/2)}\, \frac{\log j}{j^{2\, (1 - \gamma/2)\, \eps_j/(1 + \eps_j)}}.
\end{multline*}
This is impossible for large $j$ since
\[
j^{2\, (1 - \gamma/2)\, \eps_j/(1 + \eps_j)} \le j^{2\, (1 - \gamma/2)\, \eps_j} = e^{2\, (1 - \gamma/2)\, \eps_j \log j} = \O(1)
\]
by \eqref{20}.

By \ref{Lemma 6 (i)}--\ref{Lemma 6 (iii)}, $E_\mu$ has the mountain pass geometry and the mountain pass level $c_\mu$ satisfies
\[
0 < c_\mu \le \vartheta + C \mu^2 < \frac{2 \pi}{\alpha} \left(1 - \frac{\gamma}{2}\right) - \frac{\mu \theta}{2} \vol{\Omega}
\]
for all sufficiently small $\mu > 0$, so $E_\mu$ satisfies the \PS{c_\mu} condition. So $E_\mu$ has a critical point $u_\mu$ at this level by the mountain pass theorem.
\end{proof}

Next we prove the following lemma.

\begin{lemma} \label{Lemma 8}
If $\seq{u_j}$ is a convergent sequence in $H^1_0(\Omega)$, then
\[
\sup_j \int_\Omega \frac{e^{\beta u_j^2}}{|x|^\gamma}\, dx < \infty
\]
for all $\beta > 0$ and $0 \le \gamma < 2$.
\end{lemma}

\begin{proof}
Let $u \in H^1_0(\Omega)$ be the limit of $\seq{u_j}$. Since $u_j^2 \le (|u| + |u_j - u|)^2 \le 2 u^2 + 2\, (u_j - u)^2$,
\[
\int_\Omega \frac{e^{\beta u_j^2}}{|x|^\gamma}\, dx \le \left(\int_\Omega \frac{e^{4 \beta u^2}}{|x|^\gamma}\, dx\right)^{1/2} \left(\int_\Omega \frac{e^{4 \beta\, (u_j - u)^2}}{|x|^\gamma}\, dx\right)^{1/2}.
\]
The first integral on the right-hand side is finite, and the second integral equals
\[
\int_\Omega \frac{e^{4 \beta\, \norm{u_j - u}^2 w_j^2}}{|x|^\gamma}\, dx,
\]
where $w_j = (u_j - u)/\norm{u_j - u}$. Since $\norm{w_j} = 1$ and $\norm{u_j - u} \to 0$, this integral is bounded by \eqref{32}.
\end{proof}

Now we show that $u_\mu$ is positive in $\Omega$, and hence a solution of problem \eqref{1}, for all sufficiently small $\mu \in (0,\mu_0)$. It suffices to show that for every sequence $\mu_j > 0,\, \mu_j \to 0$, a subsequence of $u_j = u_{\mu_j}$ is positive in $\Omega$. By \eqref{21}, a renamed subsequence of $c_{\mu_j}$ converges to some $c$ satisfying
\[
0 < c < \frac{2 \pi}{\alpha} \left(1 - \frac{\gamma}{2}\right).
\]
Then a renamed subsequence of $\seq{u_j}$ converges in $H^1_0(\Omega)$ to a critical point $u$ of $E_0$ at the level $c$ by Theorem \ref{Theorem 2}. Since $c > 0$, $u$ is nontrivial.

Since $u_j$ is a critical point of $E_{\mu_j}$,
\[
- \Delta u_j = \lambda u_j^+\, \frac{e^{\alpha\, (u_j^+)^2}}{|x|^\gamma} + \mu_j\, \widetilde{g}(u_j)
\]
in $\Omega$. Let $2 < p < 2/\gamma$ and $1 < r < 2/\gamma p$. By the H\"{o}lder inequality,
\[
\int_\Omega \bigg|u_j^+\, \frac{e^{\alpha\, (u_j^+)^2}}{|x|^\gamma}\bigg|^p\, dx \le \left(\int_\Omega |u_j|^{pq}\, dx\right)^{1/q} \left(\int_\Omega \frac{e^{pr \alpha u_j^2}}{|x|^{\gamma pr}}\, dx\right)^{1/r},
\]
where $1/q + 1/r = 1$. The first integral on the right-hand side is bounded by the Sobolev embedding, and so is the second integral by Lemma \ref{Lemma 8} since $\gamma pr < 2$, so $u_j^+\, e^{\alpha\, (u_j^+)^2}/|x|^\gamma$ is bounded in $L^p(\Omega)$. By \eqref{2} and Lemma \ref{Lemma 8} again, $\widetilde{g}(u_j)$ is also bounded in $L^p(\Omega)$. By the Calderon-Zygmund inequality, then $\seq{u_j}$ is bounded in $W^{2,p}(\Omega)$. Since $W^{2,p}(\Omega)$ is compactly embedded in $C^1(\closure{\Omega})$ for $p > 2$, it follows that a renamed subsequence of $u_j$ converges to $u$ in $C^1(\closure{\Omega})$.

Since $u$ is a nontrivial solution of the problem
\[
\left\{\begin{aligned}
- \Delta u & = \lambda u^+\, \frac{e^{\alpha\, (u^+)^2}}{|x|^\gamma} && \text{in } \Omega\\
u & = 0 && \text{on } \bdry{\Omega},
\end{aligned}\right.
\]
$u > 0$ in $\Omega$ by the strong maximum principle and its interior normal derivative $\partial u/\partial \nu > 0$ on $\bdry{\Omega}$ by the Hopf lemma. Since $u_j \to u$ in $C^1(\closure{\Omega})$, then $u_j > 0$ in $\Omega$ for all sufficiently large $j$. This concludes the proof of Theorem \ref{Theorem 1}.

\def\cdprime{$''$}


\begin{thebibliography}{1}

\bibitem{MR2329019}
Adimurthi and K.~Sandeep.
\newblock A singular {M}oser-{T}rudinger embedding and its applications.
\newblock {\em NoDEA Nonlinear Differential Equations Appl.}, 13(5-6):585--603,
  2007.

\bibitem{MR1116249}
Ismael Ali, Alfonso Castro, and R.~Shivaji.
\newblock Uniqueness and stability of nonnegative solutions for semipositone
  problems in a ball.
\newblock {\em Proc. Amer. Math. Soc.}, 117(3):775--782, 1993.

\bibitem{MR1270096}
A.~Ambrosetti, D.~Arcoya, and B.~Buffoni.
\newblock Positive solutions for some semi-positone problems via bifurcation
  theory.
\newblock {\em Differential Integral Equations}, 7(3-4):655--663, 1994.

\bibitem{MR3507282}
Alfonso Castro, Djairo~G. de~Figueredo, and Emer Lopera.
\newblock Existence of positive solutions for a semipositone {$p$}-{L}aplacian
  problem.
\newblock {\em Proc. Roy. Soc. Edinburgh Sect. A}, 146(3):475--482, 2016.

\bibitem{MR943804}
Alfonso Castro and R.~Shivaji.
\newblock Nonnegative solutions for a class of nonpositone problems.
\newblock {\em Proc. Roy. Soc. Edinburgh Sect. A}, 108(3-4):291--302, 1988.

\bibitem{MR3406455}
M.~Chhetri, P.~Dr{\'a}bek, and R.~Shivaji.
\newblock Existence of positive solutions for a class of {$p$}-{L}aplacian
  superlinear semipositone problems.
\newblock {\em Proc. Roy. Soc. Edinburgh Sect. A}, 145(5):925--936, 2015.

\bibitem{MR3626519}
David~G. Costa, Humberto Ramos~Quoirin, and Hossein Tehrani.
\newblock A variational approach to superlinear semipositone elliptic problems.
\newblock {\em Proc. Amer. Math. Soc.}, 145(6):2661--2675, 2017.

\end{thebibliography}
\end{document}